\documentclass[11pt]{article}
\usepackage{amsthm,amsmath,amssymb}
\usepackage{natbib}
\usepackage{multirow}
\usepackage[pdftex]{graphicx}
\usepackage{subfigure}
\usepackage{makecell}
\usepackage{booktabs}
\usepackage{array}
\usepackage{fullpage}
\usepackage{url}
\usepackage{algorithm}
\usepackage{algorithmic}
\usepackage{bm}
\usepackage{smile}
\usepackage{mathtools}
\usepackage{wrapfig}
\usepackage{lipsum}
\usepackage{mathrsfs}
\usepackage{dsfont}
\usepackage{esint}

\usepackage{natbib}
\usepackage{multirow}

\usepackage{appendix}

\newcolumntype {Q}{>{$\displaystyle}l<{$}}
\newcolumntype {A}{>{$}c <{$}}

\def\tr{\mathop{\text{tr}}\kern.2ex}

\def\P{{\mathbb P}}
\def\E{{\mathbb E}}

\def\N{{\mathbb N}}
\def\Z{{\mathbb Z}}

\def\sign{\mathrm{sign}}

\long\def\comment#1{}

\def\tr{\mathop{\text{Tr}}}

\def\cS{{\mathcal{S}}}

\providecommand{\norm}[1]{\vvvert#1\vvvert}

\newcommand{\bel}{\begin{eqnarray}\label}
\newcommand{\eel}{\end{eqnarray}}
\newcommand{\bes}{\begin{eqnarray*}}
\newcommand{\ees}{\end{eqnarray*}}

\def\reals{{\mathbb{R}}}

\def\T{{\sf T}}

\usepackage[usenames,dvipsnames,svgnames,table]{xcolor}
\usepackage[colorlinks=true,
            linkcolor=blue,
            urlcolor=blue,
            citecolor=blue]{hyperref}
            
\numberwithin{equation}{section}
\numberwithin{theorem}{section}
\numberwithin{corollary}{section}
\numberwithin{asmp}{section}
\numberwithin{definition}{section}  

\begin{document}

\setlength{\abovedisplayskip}{5pt}
\setlength{\belowdisplayskip}{5pt}
\setlength{\abovedisplayshortskip}{5pt}
\setlength{\belowdisplayshortskip}{5pt}

\title{\LARGE An Exponential Inequality for U-Statistics under Mixing Conditions}

\author{Fang Han
\thanks{Department of Statistics, University of Washington, Seattle, WA 98195, USA.}}

\date{}

\maketitle

\begin{abstract}
The family of U-statistics plays a fundamental role in statistics. 
This paper proves a novel exponential inequality for U-statistics under the time series setting. 
Explicit mixing conditions are given for guaranteeing fast convergence, the bound proves to be analogous to the one under independence, and extension to non-stationary time series is straightforward. The proof relies on a novel decomposition of U-statistics via exploiting the temporal correlatedness structure. Such results are of 
interest in many fields where high dimensional time series data are present. In particular, applications to high dimensional time series inference are discussed. 
\end{abstract}


\section{Introduction}

Consider $X_1,\ldots, X_T$ to be $T$ random variables of identical distribution in a measurable space $(\cX,\cB_{\cX})$. Given a symmetric kernel function $h(\cdot):\cX^r \rightarrow \reals$, the U-statistic $U_r(X_1,\ldots, X_T)$ of order $r$ is defined as:
\begin{align}\label{eq:U-statistics}
U_r(X_1,\ldots, X_T):=\binom{T}{r}^{-1}\sum_{1\leq t_1<\cdots<t_r\leq T}h(X_{t_1},\ldots,X_{t_r}). 
\end{align}
The statistic $U_r(X_1,\ldots,X_T)$ aims to estimate 
\[
\theta(h):= \int h(X_1,\ldots,X_r)d\P(X_1)\cdots d\P(X_r). 
\]
When $r=1$, the form \eqref{eq:U-statistics} reduces to a sample-mean-type estimator. 

Under the data independence assumption, $U_r(\cdot)$ is an unbiased estimator of $\theta(h)$. \cite{hoeffding1948class} and \cite{hoeffding1963probability} further proved its asymptotic normality and characterized its tail behavior via the celebrated Bernstein and Hoeffding's inequalities for U-statistics. More recent results can be found in \cite{arcones1993limit}. 

However, in many real applications, data points are temporally correlated and naturally form a time series. Examples include stock market data, functional magnetic resonance image (fMRI) data, and time course genomic data. For such data, data independence is too strong an assumption to hold. Instead, for analysis, researchers tend to characterize the temporal dependence strength. For this, a common strategy is to assume certain types of mixing conditions. 
They do not pose any explicit structure on the time series and hence could be employed in a variety of settings \citep{bradley2005basic}. The mixing coefficients corresponding to different time series models (e.g., autoregressive (AR), moving average (MA), autoregressive-moving-average (ARMA), and copula-based models) have been calculated. We refer the readers to \cite{liebscher2005towards} and \cite{beare2010copulas} for more detailed discussions in this track. 


This paper aims to prove an exponential inequality for U-statistics under mixing conditions. Under $\alpha$- and $\beta$-mixing conditions \citep{bradley2005basic}, an emerging literature concerns characterizing the tail behaviors of sample-mean-type estimators. Related literature is as follows.
\begin{itemize}
\item
Under an $\alpha$-mixing condition, \cite{modha1996minimum}, \cite{merlevede2009bernstein}, and \cite{merlevede2011bernstein}  proved Bernstein-type inequalities for the summation of a time series of bounded or subgaussian random variables.
\item
Under a $\beta$-mixing condition, \cite{banna2015bernstein} proved a matrix Bernstein-type inequality for the summation of a matrix-valued time series of bounded singular values. 
\end{itemize}

Consider two special cases. Assume either $X_1,\ldots, X_T$ are mutually independent or the kernel function has order $r=1$. The following propositions provide tail bounds for $U_r(X_1,\ldots,X_T)-\theta(h)$. 

\begin{proposition}[\cite{hoeffding1963probability}]\label{prop:1} Assume $X_1,\ldots,X_T$ mutually independent, and there exists an absolute constant $M>0$ such that $|h(\cdot)|\leq M$. Then there exists an absolute constant $C_0$ such that, for any $x\geq 0$,
\[
\P(|U_r(X_1,\ldots,X_T)-\theta(h)|\geq x) \leq 2\exp\Big(-\frac{C_0Tx^2}{M^2+Mx}  \Big).
\]
\end{proposition}

\begin{proposition}[\cite{merlevede2009bernstein}]\label{prop:2}
Denote by $S_T:=\sum_{t=1}^T  h(X_t)$ with $|h(\cdot)|\leq M$ and $\theta(h)=0$. Suppose there exists an absolute positive constant $\gamma$ such that the sequence $\{X_t\}_{t\in\Z}$ satisfies the following $\alpha$-mixing condition\footnote{The quantity $\alpha(n)$ will be defined later in Equation \eqref{eq:alphamixing}.}:
\begin{align*}
\alpha(n)\lesssim\exp(-\gamma n).
\end{align*}
Then there exist positive absolute constants $C_1$ and $C_2$, depending only on $\gamma$, such that for $T=1,2,\ldots$ and $\eta$ satisfying $0<\eta<(C_1M(\log T)(\log\log 4T))^{-1}$, we have 
\begin{align*}
\log\E[\exp(\eta S_T)]\leq\frac{C_2\eta^2TM^2}{1-C_1\eta M(\log T)(\log\log 4T)}.
\end{align*}
In terms of probabilities, there exists an absolute constant $C_3>0$, depending only on $\gamma$, such that for $T=1,2,\ldots$ and $x\geq0$, we have
\begin{align*}
\P(|S_T|\geq x)\leq 2\exp\Big(-\frac{C_3 x^2}{TM^2+Mx(\log T)(\log\log 4T)}\Big).
\end{align*}
\end{proposition}

This paper aims to extend the above results to investigating U-statistics of order $r\geq 2$ under mixing conditions. 
The absence of such results is connected to a core problem: In the time series setting, U-statistics of order 2 or larger are fundamentally different from the sample-mean-type estimators. It is obvious that U-statistics of higher orders are no longer unbiased when the data are correlated. The bias is particularly strong when $t_1,\ldots, t_r$ are close to each other. Therefore, intuitively, U-statistics are more vulnerable to temporal dependence than the sample-mean-type statistics. In addition, technically speaking, 
as is noted in Hoeffding's seminal paper \citep{hoeffding1963probability}, to examine the tail behavior of U-statistics,  we heavily exploit the following property of U-statistics (referred to as Hoeffding's decoupling):
\begin{align}
U_r(X_1,\ldots, X_T)=\frac{1}{|\cS_T|}\sum_{\sigma\in \cS_T}\frac{1}{\lfloor T/r \rfloor}\sum_{(s_1,\ldots,s_r)\in O(\sigma)}h(X_{s_1},\ldots,X_{s_r}).
\end{align}
Here $\cS_T$ represents the set of all permutations of $\{1,\ldots,T\}$, $|\cS_T|$ represents the cardinality of $\cS_T$, $\lfloor T/r \rfloor$ represents the largest integer that is smaller than or equal to $T/r$, $(t_1,\ldots,t_T)=\sigma(1,\ldots,T)$ is a permuted series of $\{1,\ldots,T\}$, and 
\[
O(\sigma):=\Big\{(t_1,\ldots,t_r),(t_{r+1},\ldots,t_{2r}),\ldots, (t_{r\lfloor T/r \rfloor-r+1},\ldots,t_{r\lfloor T/r \rfloor})\Big\}
\] 
includes the consecutive sets of $\sigma(1,\ldots,T)$ without overlapping. 
Therefore, for analysis, we naturally desire the data to be exchangeable \citep{koroljuk1994theory}. However, the time series does not satisfy the data exchangeability property. 

This paper proves a novel exponential inequality for U-statistics, successfully addressing the above concerns. At the core of the analysis is a novel decomposition of U-statistics that is analogous to, yet fundamentally different from, the Hoeffding's decomposition approach \citep{hoeffding1948class,serfling2009approximation}. Our decomposition strategy explicitly uses the structure of the time series, could be extended to study non-stationary ones, and sheds light to understanding the mechanics of U-statistics when temporal correlatedness exists in the data. What follows is a generalization of Propositions \ref{prop:1} and \ref{prop:2}.
The tail bound proves to be analogous to the above ones. 

In the end, we provide a brief discussion on applications to high dimensional time series analysis. There we show the usefulness of the derived results on building rigorous inference. 

\subsection{Notation}
Let $\N$, $\Z$, and $\reals$ represent the sets of natural numbers, integers, and real numbers. 
For each $n\in\N$, we define $[n]=\{1,2,\ldots,n\}$. For any two real sequences $\{a_n\}$ and $\{b_n\}$, we write $a_n \lesssim b_n$, or equivalently $b_n \gtrsim a_n$, if there exists an absolute constant $C>0$ such that $|a_n|\leq C|b_n|$ for any large enough $n$. 
The term ``a.s." stands for ``almost surely". For any $x\in\reals$, we define the sign function $\sign(x):=x/|x|$, where by convention we let $0/0=0$.  For any vector $\bv\in\reals^n$ and set $I\subset [n]$, let $\bv_I$ denote the sub-vector of $\bv$ with entries indexed by $I$.  For a given matrix $\Mb=[M_{jk}]\in\reals^{n\times n}$, we write $\norm{\Mb}_{\max}:=\max_{jk}|M_{jk}|$. Throughout the paper, let $C, C'>0$ be two generic absolute constants, whose actual values may vary at different locations.

\section{Main results}

Before providing the main results derived in this paper, let's first introduce some necessary notation. For any two sigma fields $\cA$ and $\cB$ defined on a probability space $(\Omega, \cF, \P)$, we define the following three measures of dependence:
\begin{align*}
\alpha(\cA,\cB;\P)&=\sup_{A\in\cA,B\in\cB}|\P(A\cap B)-\P(A)\P(B)|,\\
\phi(\cA,\cB;\P)&=\sup_{A\in\cA,B\in\cB,\P(A)>0}|\P(B|A)-\P(B)|,\\
\beta(\cA,\cB;\P)&=\sup\frac{1}{2}\sum_{i=1}^I\sum_{j=1}^J |\P(A_i\cap B_j)-\P(A_i)\P(B_j)|,
\end{align*}
where the supremum is taken over all pairs of partitions of $\Omega$ such that $A_i\in\cA$ and $B_j\in\cB$. 

Let $\{X_t\}_{t\in\Z}$ be the sequence of random variables of interest defined on the probability space $(\Omega, \cF, \P)$. For $-\infty\leq J\leq L\leq \infty$, define the sigma field generated by $\{X_t; J\leq t\leq L\}$ as follows:
\begin{align*}
\cF_{J}^{L}=\sigma(X_t;J\leq t\leq L,t\in\Z).
\end{align*}
For each $n\geq1$, define the three dependence coefficients corresponding to the sequence $\{X_t\}_{t\in\Z}$:
\begin{align}\label{eq:alphamixing}
\alpha(n)=\sup_{j\in\Z}\alpha(\cF_{-\infty}^j,\cF_{j+n}^{\infty};\P),~ \phi(n)=\sup_{j\in\Z}\phi(\cF_{-\infty}^j,\cF_{j+n}^{\infty};\P),~\beta(n)=\sup_{j\in\Z}\beta(\cF_{-\infty}^j,\cF_{j+n}^{\infty};\P).
\end{align}
The sequence $\{X_t\}_{t\in\Z}$ is said to be $\alpha$-(or $\phi$-, or $\beta$-)mixing if $\alpha(n)$ (or $\phi(n)$, or $\beta(n)$)
tends to zero as $n\rightarrow \infty$. In addition, if $X_t$ is independent of $X_s$ as long as $|t-s|>m$, $\{X_t\}_{t\in\Z}$ is said to be $m$-dependent. It is obvious $\alpha(n)\leq\beta(n)\leq\phi(n)$, and the corresponding mixing conditions are all weaker than $m$-dependence.


\subsection{Main theorem}


For analysis, we require the following three assumptions on the time series and U-statistics:
\begin{itemize}
\item {\bf Assumption (A1).} Assume $\{X_t\}_{t\in\Z}$ is strictly stationary, and there exists an absolute constant $\delta\geq1$ such that for any $n\geq 1$, we have the $\beta$-mixing coefficient, corresponding to $\{X_t\}_{t\in\Z}$, satisfies
$\beta(n) \lesssim n^{-\delta}. $
\item {\bf Assumption (A2).} Assume, uniformly, for any integer $J$ such that $1\leq J\leq r-1$ and arbitrary $1\leq t_1<\cdots< t_J \leq T$, conditional on $X_{t_1},\ldots,X_{t_{J}}$,  the sequence $\{X_t\}_{t=t_J+1}^{\infty}$ satisfies, for the $\alpha$-mixing coefficient corresponding to it, 
\[
\alpha(n; X_{t_1},\ldots,X_{t_J}):=\sup_{j\geq t_J+1}\alpha(\cF_{t_J+1}^j,\cF_{j+n}^{\infty};\P(\cdot|X_{t_1},\ldots,X_{t_J})) \lesssim \exp(-\gamma n), ~~{\rm a.s.},
\]
where $\P(\cdot|X_{t_1},\ldots,X_{t_J})$ stands for the conditional probability\footnote{In other words, there exists an absolute constant $C>0$ such that $\P\big(\alpha(n; X_{t_1},\ldots,X_{t_J})\leq C\exp(-\gamma n)\big)=1$ for any large enough $n$.}. In particular, we have, for the $\alpha$-mixing coefficient corresponding to $\{X_t\}_{t\in\Z}$ itself, 
\[
\alpha(n) \lesssim \exp(-\gamma n).
\]
\item {\bf Assumption (A3).} The kernel function $h(\cdot)$ is symmetric\footnote{Note any asymmetric kernel function could be converted to a symmetric one \citep{serfling2009approximation}. Therefore, the kernel symmetry is not a constraint. } and there exists a positive absolute constant $M$ such that  $|h(x_1,\ldots,x_r)|\leq M$ for any $(x_1,\ldots,x_r)\in\cX^r$.
\end{itemize}


Under Assumptions {\bf (A1)-(A3)}, our main theorem, given below, is a generalization of Propositions \ref{prop:1} and \ref{prop:2}. 

\begin{theorem}\label{thm:main} Let $\{X_t\}_{t\in\Z}$ be a data sequence, along with the kernel function $h(\cdot)$, satisfying Assumptions {\bf (A1), (A2),} and {\bf (A3)}. We then have, there exist absolute constants $C_4, C_5>0$ only depending on $\gamma$ and $r$, such that, for any $x\geq 0$ and $T$ sufficiently large,
\begin{align*}
\P(|U_r(X_1,\ldots,X_T)-\theta(h)|\geq C_4M/\sqrt{T}+x)\leq2\exp\Big(-\frac{C_5x^2T}{M^2+Mx(\log T)(\log\log 4T)}\Big).
\end{align*}
\end{theorem}


In the following, we first provide several remarks on the assumptions we posed. 

\begin{remark}
Assumption {\bf(A1)} is added only for proving that the bias $\E U_r(X_1,\ldots,X_T)-\theta(h)$ is small. Later, without requiring {\bf (A1)}, Theorem \ref{thm:variance} will show $U_r(X_1,\ldots,X_T)$ converges to $\E U_r(X_1,\ldots,X_T)$ exponentially fast. 
\end{remark}
\begin{remark}
The $\beta$-mixing condition {\bf (A1)} is typically required in obtaining asymptotic normality for U-statistics, when no stringent Lipchitz-continuity assumption on the kernel functions is posed \citep{yoshihara1976limiting,dehling2010central}. Via assuming kernel boundedness, we could give a better $\beta$-mixing decaying rate than \cite{yoshihara1976limiting}.  Of note, the rate in Assumption {\bf (A1)} is attainable in many situations. For example,  \cite{longla2012some} and \cite{longla2013remarks} provided sufficient conditions under which the Markov chain is $\beta$-mixing with exponentially decaying rate. We also refer the readers to \cite{mokkadem1990proprietes}  for similar results in ARMA models. 
\end{remark}
\begin{remark}
Assumption {\bf (A2)} is weaker than the $\phi$-mixing condition in many cases. As will be shown in the next section, finite-state and vector-valued absolutely continuous data sequences of exponentially $\phi$-mixing decaying rate satisfy {\bf (A2)}. 
Accordingly, Assumption {\bf (A2)} holds for many copula-based Markov chains \citep{longla2012some}. 
\end{remark}
\begin{remark}
In contrast to the results in \cite{yoshihara1976limiting}, in Assumption {\bf (A2)}, we require an exponentially, instead of polynomially, mixing decaying rate. This is because, for obtaining sharp concentration inequality, compared to \cite{yoshihara1976limiting} and \cite{dehling2010central}, we need to calculate higher moments of $U_r(\cdot)$. The ``exponentially decaying rate" condition is routine in the literature of deriving concentration inequalities for weakly dependent data. For this, we refer the readers to \cite{merlevede2009bernstein}, \cite{merlevede2011bernstein}, and the arguments therein. 
\end{remark}

Secondly, we compare Theorem \ref{thm:main} to Theorem 2 in \cite{borisov2009exponential}, which to our knowledge is the only relevant exponential concentration inequality for U-statistics under mixing conditions. \cite{borisov2009exponential} aimed to study the tail behaviors of canonical (degenerate) U-statistics under the $\phi$-mixing condition. There are three main observations: 
(i) The result in Theorem \ref{thm:main} can be extended to study non-stationary time series (see Theorem \ref{thm:variance} for details), while the results in \cite{borisov2009exponential} cannot; (ii) the orthogonal expansion of kernel functions and related conditions in Theorem 2 therein are difficult to verify and interpret. In comparison, the exponential inequality obtained in this paper is clear and easy to use; and (iii) the proof strategy built in this paper, based on a novel decomposition of U-statistics, is fundamentally different from theirs, which is built on the classic treatment to U-statistics: linear expansions to kernels and decoupling kernels to sample-mean-type statistics \citep{arcones1993limit}.

\subsection{Discussions on Assumption (A2)}

This section discusses the rationality of Assumption {\bf (A2)}. First, we prove that Assumption {\bf (A2)} is weaker than the finite-state or vector-valued absolutely continuous $\phi$-mixing condition.

\begin{proposition}\label{prop:3}
Let $\{X_t\}_{t\in\Z}$ be a sequence of random variables with each value in a finite set $G_t$. Suppose $\{X_t\}_{t\in\Z}$ satisfies the $\phi$-mixing condition:
\begin{align*}
\phi(n)\lesssim\exp(-C n),~~{\rm for}~n\geq 1, 
\end{align*}
where $C$ is an absolute positive constant. Then, uniformly, for any integer $J$ such that $1\leq J\leq r-1$ and arbitrary $1\leq t_1<\ldots<t_J\leq T$, conditional on $X_{t_1},\ldots,X_{t_J}$, the sequence $\{X_t\}_{t=t_J+1}^{\infty}$ satisfies, for the $\phi$-mixing coefficient corresponding to it, 
\begin{align}\label{eq:phi}
\phi(n;X_{t_1},\ldots,X_{t_J}):=\sup_{j\geq t_J+1}\phi(\cF_{t_J+1}^j,\cF_{j+n}^{\infty};\P(\cdot|X_{t_1},\ldots,X_{t_J})) \lesssim \exp(-Cn), ~~~{\rm a.s.}.
\end{align}
\end{proposition}

Of note, such a finite-state $\phi$-mixing sequence could always be obtained via ``truncating" a $\phi$-mixing sequence. To see this, let $\{X_t\}_{t\in\Z}$ be a sequence of random variables defined on $(\Omega,\cF,\P)$ with each mapping to a measurable space $(\cX,\cB_{\cX})$. Let $P:=\{A_1,\ldots,A_l\}$ be a finite measurable partition of $\cX$ and $\{a_1,\ldots,a_l\}$ be a set of arbitrary $l$ different objects. Define $\{\bar{X_t}\}_{t\in \Z}$ to be random variables satisfying that $\bar{X_t}=a_i$ if $X_t\in A_i$ for $i=1,\ldots,l$ and $t\in\Z$. If $\{X_t\}_{t\in\Z}$ satisfies the $\phi$-mixing condition with
\begin{align}\label{eq:han-2}
\phi(n)\lesssim\exp(-C n),
\end{align}
where $C$ is an absolute positive constant, we have $\{\bar{X}_t\}_{t\in\Z}$ also satisfies the same $\phi$-mixing condition. Then, using Proposition \ref{prop:3}, uniformly, for any integer $J$ such that $1\leq J\leq r-1$ and arbitrary $1\leq t_1<\ldots<t_J\leq T$, conditional on $\bar{X}_{t_1},\ldots,\bar{X}_{t_J}$,  the sequence $\{\bar{X}_t\}_{t=t_J+1}^{\infty}$ satisfies the $\phi$-mixing condition in \eqref{eq:phi} almost surely.  

We then show the vector-valued absolutely continuous $\phi$-mixing time series also satisfies Assumption {\bf (A2)}. 

\begin{proposition}\label{prop:new1} The conclusion in Proposition \ref{prop:3} also holds for vector-valued absolutely continuous time series.
\end{proposition}

For more discussions on the time series satisfying different $\phi$-mixing conditions, we refer the readers to  \cite{bradley2005basic}, \cite{longla2012some}, and references therein. In particular, of note, \cite{longla2012some} showed a variety of stationary copula-based Markov chains satisfy the $\phi$-mixing condition in \eqref{eq:han-2}. 

Secondly, we show that any $m$-dependent sequence $\{X_t\}_{t\in\Z}$ satisfies Assumption {\bf (A2)}. 

\begin{proposition}
Let $\{X_t\}_{t\in\Z}$ be a sequence of random variables satisfying the $m$-dependence condition. Then, for any integer $J$ such that $1\leq J\leq r-1$ and arbitrary $1\leq t_1<\ldots<t_J\leq T$, conditional on $X_{t_1},\ldots,X_{t_J}$, the sequence $\{X_t\}_{t=t_J+1}^{\infty}$ satisfies the $m$-dependence condition. In other words, for arbitrary $t,s\geq t_J+1$, $X_t$ is conditionally independent of $X_s$ given $X_{t_1},\ldots,X_{t_J}$, as long as $|t-s|>m$.
\end{proposition}
\begin{proof}
For any integer $J$ such that $1\leq J\leq r-1$ and arbitrary $1\leq t_1<\ldots<t_J\leq T$, due to the $m$-dependence condition for $\{X_t\}_{t\in\Z}$, we have, for all $j\geq1$, $(X_{t_1},\ldots,X_{t_J}, X_{t_J+1},\ldots,X_{t_J+j})$ and $(X_{t_J+j+m+1},\ldots)$ are independent. Using Lemma \ref{lem:condind}, we conclude that, conditional on $X_{t_1},\ldots,X_{t_J}$,  the sequence $\{X_t\}_{t=t_J+1}^{\infty}$ satisfies the $m$-dependence condition. 
\end{proof}

\subsection{Applications to high-dimensional statistical inference}

This section considers two specific examples of U-statistics, Kendall's tau and Spearman's rho, that have been heavily exploited in robust inference of large graphical models \citep{liu2012high}, covariance matrices \citep{zhao2014positive}, and transition matrices in copula-based Markov chains \citep{han2015rate}.
\begin{example}[Kendall's tau] For a given data sequence $\bX_1,\ldots,\bX_T\in\reals^2$ of $\bX_{t}:=(X_{t1},X_{t2})^\T$, the Kendall's tau correlation coefficient is defined as 
\[
\tau_2(\bX_1,\ldots, \bX_T):=\frac{2}{T(T-1)}\sum_{t<t'}\sign(X_{t1}-X_{t'1})\sign(X_{t2}-X_{t'2}).
\]
It is a U-statistic of order 2. 
\end{example}
\begin{example}[Spearman's rho] Spearman's rho is defined as the correlations of the ranks of the data. It is not a U-statistic. However, by \cite{hoeffding1948class}, we can rewrite the Spearman's rho correlation coefficient estimator $\rho$ as follows:
\begin{align}\label{eq:rho}
\rho=\frac{T-2}{T+1}\cdot\underbrace{\frac{3}{T(T-1)(T-2)}\sum_{t\ne t'\ne t''}\sign(X_{t1}-X_{t'1})\sign(X_{t2}-X_{t''2})}_{\rho_3(\bX_1,\ldots,\bX_T)}+\frac{3\tau_2(\bX_1,\ldots,\bX_T)}{T+1},
\end{align}
where $\rho_3(\bX_1,\ldots,\bX_T)$ is a U-statistic of order 3 and an asymmetric kernel function, which could be rewritten as a U-statistic of a symmetric kernel \citep{hoeffding1963probability}. 
\end{example}

There are immediate consequences of Theorem \ref{thm:main}. For example, given a high dimensional time series $\bX_1,\ldots,\bX_T\in\reals^p$, we define the Kendall's tau and Spearman's rho correlation matrix estimators as $\hat\Tb=[\tau_{jk}]$ and $\hat\Sbb=[\rho_{jk}]$ with $\tau_{jk}$ and $\rho_{jk}$ defined as above. For chatacterizing the impact of dimensionality in analysis,  we allow the dimension $p=p_T$ to increase with the time series length $T$.

Let $\Tb$ and $\Sbb$ be the population counterparts to $\hat\Tb$ and $\hat\Sbb$ under independence. The following corollary quantifies the largest marginal deviations of $\hat\Tb$ and $\hat\Sbb$ to $\Tb$ and $\Sbb$. 

\begin{corollary}\label{cor:example} Suppose that the conditions in Theorem \ref{thm:main} hold for any times series $\{(\bX_t)_{\{j,k\}}\in\reals^2\}_{t\in \Z}$ of $j\ne k\in [p]$. We then have, 
\[
\norm{\hat\Tb-\Tb}_{\max}=O_P(\sqrt{\log (Tp) /T})~~{\rm and}~~\norm{\hat\Sbb-\Sbb}_{\max}=O_P(\sqrt{\log (Tp)/T}).
\]
\end{corollary}
\begin{proof}
We only prove the case for Spearman's rho since the proof for Kendall's tau is similar. For showing $\norm{\hat\Sbb-\Sbb}_{\max}=O_P(\sqrt{\log (Tp)/T})$, we employ the union bound argument. In detail, for any $x>0$, we have 
\begin{align}\label{eq:newhan-4}
\P(\norm{\hat\Sbb-\Sbb}_{\max}>x+C/\sqrt{T})\leq \frac{p(p-1)}{2}\max_{j,k}\P\Big(|\hat\Sbb_{jk}-\Sbb_{jk}|>x+C/\sqrt{T}\Big).
\end{align}
Notice that: (i) in \eqref{eq:rho}, $|3\tau_2((\bX_1)_{\{j,k\}},\ldots,(\bX_T)_{\{j,k\}})/(T+1)| \leq 3/T$; (ii) $\rho_3((\bX_1)_{\{j,k\}},\ldots,(\bX_T)_{\{j,k\}})$ has a kernel upper bounded by an absolute constant $M_0$. By picking $C$ in \eqref{eq:newhan-4} large enough, Theorem \ref{thm:main} then yields
\[
\P(\norm{\hat\Sbb-\Sbb}_{\max}>x+C/\sqrt{T})\leq p^2\exp\Big(-\frac{C'x^2T}{M_0^2+M_0x(\log T)(\log\log 4T)}\Big).
\]
Picking $x=(3M_0^2\log (Tp)/(C'T))^{1/2}$, we have $\P(\norm{\hat\Sbb-\Sbb}_{\max}>x+C/\sqrt{T})\to 0$, which completes the proof.
\end{proof}

Results in Corollary \ref{cor:example} are key in high dimensional statistical inference. It forms the basis for studying sparse matrix \citep{bickel2008covariance},  sparse inverse matrix \citep{ravikumar2011high}, sparse principal component \citep{yuan2013truncated}, sparse transition matrix \citep{han2015rate} estimators, to just name a few.

\section{Proofs}


Before proving Theorem \ref{thm:main}, let's first denote the expectation of $U_r(X_1,\ldots,X_T)$ as 
\begin{align*}
\theta^*(h)=\binom{T}{r}^{-1}\sum_{1\leq t_1<\ldots<t_r\leq T}\E[h(X_{t_1},\ldots,X_{t_r})].
\end{align*}
By simple observation, we have $\theta^*(h)=\theta(h)$ under data independence assumption. However, when $X_1,\ldots,X_T$ are not mutually independent, $\theta^*(h)$ is not necessarily equal to $\theta(h)$. Therefore, we decompose $U_r(X_1,\ldots,X_T)-\theta(h)$ into two parts:
\begin{align}\label{eq:decomp}
U_r(X_1,\ldots,X_T) -\theta(h) = \underbrace{U_r(X_1,\ldots,X_T)-\theta^*(h)}_{\rm variance~term}+\underbrace{\theta^*(h)-\theta(h)}_{\rm bias~term}.
\end{align}
In the following, we separately bound the variance and bias terms.

\subsection{Bounding the variance term}

For bounding the variance term in \eqref{eq:decomp}, we need to carefully decouple the terms $\{h(X_{t_1},\ldots, X_{t_r}); 1\leq t_1<\cdots<t_r\leq T\}$. We first state the main theorem as follows. Of note, Theorem \ref{thm:variance} holds without requiring the time series to be strictly stationary. 

\begin{theorem}[Variance] \label{thm:variance}
Assume Assumptions {\bf (A2)} and {\bf (A3)} hold. We then have, there exist absolute positive constants $C_6, C_7$, only depending on $\gamma$ and $r$, such that for  arbitrary $0<\eta<(C_6M T^{-1}(\log T)(\log\log 4T))^{-1}$, 
\begin{align*}
&\log\E(\exp[\eta\{U_r(X_1,\ldots,X_T)-\theta^*(h)\}])\leq \frac{C_7\eta^2M^2 T^{-1}}{1-C_6\eta M T^{-1}(\log T)(\log\log 4T)}.  
\end{align*}
In terms of probabilities, there exists an absolute constant $C_5$, defined in Theorem \ref{thm:main}, depending only on $\gamma$ and $r$, such that for all $x\geq 0$ and $T$ sufficiently large, we have
\begin{align*}
\P(|U_r(X_1,\ldots,X_T)-\theta^*(h)|\geq x)\leq 2\exp\Big(-\frac{C_5x^2T}{M^2+Mx(\log T)(\log\log 4T)}\Big).
\end{align*}
\end{theorem}

The proof of Theorem \ref{thm:variance} is lengthy. For better presenting the main idea behind the proof, we first give a very brief proof sketch. Intrinsically, there are two main steps in the proof. The first step is a novel decomposition of the U-statistic $U_r(X_1,\ldots, X_T)$. In detail, we decompose the summation 
\[
\binom{T}{r}^{-1}\sum_{1\leq t_1<\cdots<t_r\leq T} h(X_{t_1},\ldots, X_{t_r})
\]
into many parts. In each part, the randomness is only posed on the ``present time", while the ``history" is conditioned on and the ``future" is integrated out. The aim is to employ the exponential inequality constructed in \cite{merlevede2009bernstein} for each part. 

In the second step, noting that each component in the decomposition has different converging rate to zero, ranging from a constant rate to a fast exponential rate, we summarize these components together using an improved Jensen's inequality. It guarantees the components of slowly decaying rates are asymptotically ignorable. The final rate proves to be analogous to that under the independence setting. 

The detailed proof is as follows.




\begin{proof}
By the definition of $\theta^*(h)$, we have
\begin{align*}
&U_r(X_1,\ldots, X_T)- \theta^*(h)=\\
&\binom{T}{r}^{-1}\sum_{1\leq t_1<\cdots<t_r\leq T}\Big\{h(X_{t_1},\ldots,X_{t_r})-\E[h(X_{t_1},\ldots,X_{t_r})]\Big\}.
\end{align*}
For presentation clearness, for each selected $\{t_1,\ldots,t_r\}$ such that $1\leq t_1<\cdots<t_r\leq T$, in the sequel, we give notation for the following (conditional) expectations:
\begin{align*}
\theta_{[t_1:t_r]}&:=\E\{h(X_{t_1},\ldots,X_{t_r})\},\\
\hat\theta_{[t_1:t_r]}^{[t_1:t_{r-k}]}&:=\E\{h(X_{t_1},\ldots,X_{t_r})|X_{t_1},\ldots,X_{t_{r-k}}\},~~~{\rm for}~~k=1,\ldots,r-1.
\end{align*}
Here $\hat\theta_{[t_1:t_r]}^{[t_1:t_{r-k}]}$ stands for a functional of $X_{t_1},\ldots,X_{t_{r-k}}$. Thus, we can first decompose $U_r(X_1,\ldots, X_T)- \theta^*(h)$ as follows:
\begin{align*}
U_r(X_1,\ldots, X_T)- \theta^*(h)=S_1+S_2+\ldots+S_{r-1}+S_r,
\end{align*}
where $S_1, S_r$ are defined as
\[
S_1:=\binom{T}{r}^{-1}\sum_{1\leq t_1<\cdots<t_r\leq T}\Big\{\big(h(X_{t_1},\ldots,X_{t_r})-\theta_{[t_1:t_r]}\big)-\big(\hat\theta_{[t_1:t_r]}^{[t_1:t_{r-1}]}-\theta_{[t_1:t_r]}\big)\Big\},
\]
and
\[
S_r:=T^{r-1}\cdot\binom{T}{r}^{-1} \sum_{1\leq t_1<\cdots<t_r\leq T}\Big\{T^{-(r-1)}\big(\hat\theta_{[t_1:t_r]}^{[t_1]}-\theta_{[t_1:t_r]}\big)\Big\},
\]
and for $k=2,\ldots,r-1$, $S_k$ is defined as
\[
S_k := T^{k - 1}\binom{T}{r}^{-1}         \sum_{1\leq t_1<\cdots<t_r\leq T}   \Big\{T^{-(k - 1)}\Big
[\big(\hat\theta_{[t_1:t_r]}^{[t_1:t_{r-k+1}]} - \theta_{[t_1:t_r]}\big)-\big(\hat\theta_{[t_1:t_r]}^{[t_1:t_{r-k}]} - \theta_{[t_1:t_r]}\big)\Big]\Big\}.
\]
In the following, we define $\{A_{[t_1:t_r]}^{[t_1:t_{r-k}]}; 1\leq t_1<\cdots<t_r\leq T\}$ to be the terms within each $S_k$ such that 
\[
S_k=T^{k-1}\cdot\binom{T}{r}^{-1}\sum_{1\leq t_1<\cdots<t_r\leq T} A_{[t_1:t_r]}^{[t_1:t_{r-k}]}
\]
for $k=1,\ldots,r-1$. Similarly, we define $\{A_{[t_1:t_r]}^{[\cdot]}; 1\leq t_1<\cdots<t_r\leq T\}$ to be the terms within $S_r$ such that 
\[
S_r=T^{r-1}\cdot\binom{T}{r}^{-1}\sum_{1\leq t_1<\cdots<t_r\leq T}A_{[t_1:t_r]}^{[\cdot]}. 
\]
More specifically, for $2\leq k\leq r-1$, the definitions of $A_{[\cdots]}^{[\cdots]}$ are as follows:
\begin{align*}
&A_{[t_1:t_r]}^{[t_1:t_{r-1}]}:=\big(h(X_{t_1},\ldots,X_{t_r})-\theta_{[t_1:t_r]}\big)-\big(\hat\theta_{[t_1:t_r]}^{[t_1:t_{r-1}]}-\theta_{[t_1:t_r]}\big),\\
&A_{[t_1:t_r]}^{[t_1:t_{r-k}]}:=T^{-(k - 1)}\cdot\Big(\big(\hat\theta_{[t_1:t_r]}^{[t_1:t_{r-k+1}]} - \theta_{[t_1:t_r]}\big) - \big(\hat\theta_{[t_1:t_r]}^{[t_1:t_{r-k}]} - \theta_{[t_1:t_r]}\big)\Big),\\
{\rm and}~~~&A_{[t_1:t_r]}^{[\cdot]}:=T^{-(r - 1)}\cdot\big(\hat\theta_{[t_1:t_r]}^{[t_1]}-\theta_{[t_1:t_r]}\big).
\end{align*}
We further define
\begin{align*}
&B_{[t_1:t_r]}^{[t_1:t_{r-1}]}:=A_{[t_1:t_r]}^{[t_1:t_{r-1}]},\\
&B_{[t_1:t_r]}^{[t_1:t_{r-k}]}:= \sum_{t_{r-k+2}=t_{r-k+1}+1}^{T-k+2}\ldots\sum_{t_r=t_{r-1}+1}^T  A_{[t_1:t_r]}^{[t_1:t_{r-k}]},~~{\rm for}~k=2,\ldots,r-1,\\
{\rm and}~~~&B_{[t_1:t_r]}^{[\cdot]}:=\sum_{t_2=t_1+1}^{T-r+2}\ldots\sum_{t_r=t_{r-1}+1}^T A_{[t_1:t_r]}^{[\cdot]}.
\end{align*}
Note, conditional on $X_{t_1},\ldots,X_{t_{r-1}}$, $B_{[t_1:t_r]}^{[t_1:t_{r-1}]}$ is only a random variable as a function of $X_{t_r}$. For $k=2,\ldots,r-1$, conditional on $X_{t_1},\ldots,X_{t_{r-k}}$, $B_{[t_1:t_r]}^{[t_1:t_{r-k}]}$ is only a random variable with regard to $X_{t_{r-k+1}}$. And in the end, $B_{[t_1:t_r]}^{[\cdot]}$ is a random variable only relevant to $X_{t_1}$. 

With the above definitions, for $1\leq k\leq r-1$, we can reorganize the summation in $S_k$ as follows:
\begin{align}\label{eq:han-1}
S_k&=T^{k-1} \binom{T}{r}^{- 1}   \sum_{t_{r - k}=r - k}^{T-k}\Big\{   \sum_{1\leq t_1<\cdots<t_{r-k-1}\leq t_{r - k} - 1}  \Big( \sum_{t_{r - k+1}=t_{r-k}+1}^{T-k+1} B_{[t_1:t_r]}^{[t_1:t_{r-k}]}\Big)\Big\}, \\
~{\rm and}~~S_r&=T^{r-1}\binom{T}{r}^{-1}\sum_{t_1=1}^{T-r+1}B_{[t_1:t_r]}^{[\cdot]}. \notag
\end{align}
A further observation verifies the following three properties:
 
 {\bf Property (P1).} We have the following (conditional) expectations are all equal to zero almost surely:
\begin{align*}
\E\Big(B_{[t_1:t_r]}^{[t_1:t_{r-k}]}\Big|X_{t_1},\ldots,X_{t_{r-k}}\Big)&=0, ~~~{\rm for}~~k=1,\ldots,r-1,\\
\E\Big(B_{[t_1:t_r]}^{[\cdot]}\Big)&=0.
\end{align*}

{\bf Property (P2).} Using Assumption {\bf (A3)}, we have $B_{[\cdots]}^{[\cdots]}$ are all bounded:
\begin{align*}
\Big|B_{[t_1:t_r]}^{[t_1:t_{r-k}]}\Big|&\leq 2M, ~~~{\rm for}~~k=1,\ldots,r-1,\\
~{\rm and}~~~\Big|B_{[t_1:t_r]}^{[\cdot]}\Big|&\leq2M.
\end{align*}

{\bf Property (P3).} For fixed $k\in\{1,\ldots,r-1\}$, fixing $t_1,\ldots,t_{r-k}$, conditional on $X_{t_{1}},\ldots,$ $X_{t_{r-k}}$, using Assumption {\bf (A2)}, we have the sequence
\[
\Big\{B_{[t_1:t_r]}^{[t_1:t_{r-k}]}; t_{r-k+1}=t_{r-k}+1,\ldots, T-k+1\Big\}
\]
satisfies the $\alpha$-mixing condition with
\begin{align*}
\alpha(n; X_{t_1},\ldots,X_{t_{r-k}}) \lesssim \exp(-\gamma n), ~~{\rm a.s.}.
\end{align*}
Similarly, we have $\Big\{B_{[t_1:t_r]}^{[\cdot]}\Big\}_{t_1=1}^{T-r+1}$ satisfies the $\alpha$-mixing condition with
\begin{align*}
\alpha(n) \lesssim \exp(-\gamma n).
\end{align*}

Employing the above three properties, we then turn to bound the logarithmic Laplace transform, $\log\E\exp(\eta S_k)$, for $k=1,\ldots,r$. We separate the proof into two parts. The first part is focused on the cases $k=1,\ldots,r-1$. The second part is focused on the case $k=r$. 

{\bf Step I.} First, for $k=1,\ldots,r-1$, we divide the proof into three steps.

{\bf Step I.1.} Reminding the decomposition of $S_k$ in \eqref{eq:han-1}, we first aim at bounding the term 
\[
\sum_{t_{r-k+1}=t_{r-k}+1}^{T-k+1}B_{[t_1:t_r]}^{[t_1:t_{r-k}]}
\]
within each $S_k$, via first fixing the history of $X_{t_1},\ldots, X_{t_{r-k}}$. Also note 
\[
\sum_{t_{r-k+1}=t_{r-k}+1}^{T-k+1}B_{[t_1:t_r]}^{[t_1:t_{r-k}]} 
\]
has integrated out the ``future" of $X_{t_{r-k+2}},\ldots, X_{t_r}$. Using the law of iterated expectations, we have
\begin{align*}
&\log\E\Big[\exp\Big(\eta \sum_{t_{r-k+1}=t_{r-k}+1}^{T-k+1}B_{[t_1:t_r]}^{[t_1:t_{r-k}]}\Big)\Big]\\
=&\log\E\Big[\E\Big[\exp\Big(\eta \sum_{t_{r-k+1}=t_{r-k}+1}^{T-k+1}B_{[t_1:t_r]}^{[t_1:t_{r-k}]}\Big)\Big|X_{t_1},\ldots,X_{t_{r-k}}\Big]\Big].
\end{align*}
Exploiting Properties {\bf (P1)}, {\bf (P2)}, {\bf (P3)}, and Proposition \ref{prop:2}, for all $\eta$ satisfying 
\[
0<\eta<\Big[2C_1 M\log(T-t_{r-k}-k+1)\log\log4(T-t_{r-k}-k+1)\Big]^{-1},
\]
we have 
\begin{align*}
&\log\E\Big[\exp\Big(\eta \sum_{t_{r-k+1}=t_{r-k}+1}^{T-k+1}B_{[t_1:t_r]}^{[t_1:t_{r-k}]}\Big)\Big]\\
 \leq&\log\E\Big[\exp\Big(\frac{4C_2\eta^2M^2(T - t_{r-k} - k + 1)}{1 - 2C_1\eta M\log(T - t_{r-k} - k + 1)\log\log4(T - t_{r-k} - k + 1)}\Big)\Big]\\
 =&\frac{4C_2\eta^2M^2(T - t_{r-k} - k + 1)}{1 - 2C_1\eta M\log(T - t_{r-k} - k + 1)\log\log4(T - t_{r-k} - k + 1)}.
\end{align*} 

{\bf Step I.2.} Secondly, we try to bound the summation over different sequences of the history $t_1,\ldots, t_{r-k-1}$ for $\sum_{t_{r - k + 1}=t_{r - k} + 1}^{T - k + 1}B_{[t_1:t_r]}^{[t_1:t_{r - k}]}$. For this, using the improved Jensen's inequality in  Lemma \ref{lem:combn},  for all $\eta$ satisfying 
\[
0<\eta<\Big[2C_1 M\log(T-t_{r-k}-k+1)\log\log4(T-t_{r-k}-k+1)\cdot\binom{t_{r-k}-1}{r-k-1}\Big]^{-1},
\]
we have
\begin{align}\label{eq:newhan-1}
&\log\E\Big[\exp\Big(\eta \sum_{1\leq t_1<\cdots<t_{r-k-1}\leq t_{r-k}-1}\sum_{t_{r-k+1}=t_{r-k}+1}^{T-k+1}B_{[t_1:t_r]}^{[t_1:t_{r-k}]}\Big)\Big]\notag\\
\leq&\frac{4C_2\eta^2M^2(T - t_{r - k} - k + 1)\binom{t_{r-k}-1}{r-k-1}^2}{1 - 2C_1\eta M\log(T-t_{r-k}-k+1)\log\log4(T-t_{r-k}-k+1)\cdot\binom{t_{r-k}-1}{r-k-1}}.
\end{align}

{\bf Step I.3.} Finally, we turn to analyze the summation over all different choices of $t_{r-k}$ and derive a bound of $\log\E\exp(\eta S_k)$. 
Equation \eqref{eq:han-1} has shown that $$T^{-k+1}\binom{T}{r}\cdot S_k$$ is a summation over the items in \eqref{eq:newhan-1}. Accordingly, using Lemma \ref{lem:combn} and the bound in \eqref{eq:newhan-1} suffices to obtain the desired result. 

In detail, let's define $\sigma_{t_{r-k}}^{(k)}$ for $k=1,\ldots,r-1$ and $t_{r-k}=r-k,\ldots,T-k$ to be
\begin{align*}
\sigma_{t_{r-k}}^{(k)}=(T-t_{r-k}-k+1)^{\frac{1}{2}}\binom{t_{r-k}-1}{r-k-1}.
\end{align*}
Similarly, we define $\kappa_{t_{r-k}}^{(k)}$ for $k=1,\ldots,r-1$ and $t_{r-k}=r-k,\ldots,T-k$ to be
\begin{align*}
\kappa_{t_{r-k}}^{(k)}=\log(T-t_{r-k}-k+1)\log\log4(T-t_{r-k}-k+1)\cdot\binom{t_{r-k}-1}{r-k-1}.
\end{align*}
The numbers $\sigma_{t_{r-k}}^{(k)}$ and $\kappa_{t_{r-k}}^{(k)}$ correspond to the coefficients in \eqref{eq:newhan-1} that change with the value of $t_{r-k}$. We define $\sigma^{(k)}$ for $k=1,\ldots,r-1$ as follows 
\begin{align*}
\sigma^{(k)}&=\sum_{t_{r-k}=r-k}^{T-k}\sigma_{t_{r-k}}^{(k)},
\end{align*}
and it has the upper bound
\begin{align}\label{eq:newhan-2}
\sum_{t_{r-k}=r-k}^{T-k}\sigma_{t_{r-k}}^{(k)}\lesssim\sum_{t_{r-k}=r-k}^{T-k}t_{r-k}^{r-k-1}T^{\frac{1}{2}}\lesssim\int_{0}^T x^{r-k-1}T^{\frac{1}{2}} dx\lesssim T^{r-k+\frac{1}{2}}.
\end{align}

Similarly, we define $\kappa^{(k)}$ for $k=1,\ldots,r-1$ as follows
\begin{align*}
\kappa^{(k)}& = \sum_{t_{r - k}=r - k}^{T- k}\kappa_{t_{r-k}}^{(k)},
\end{align*}
and it has the upper bound
\begin{align}\label{eq:newhan-3}
\sum_{t_{r-k}=r-k}^{T-k}\kappa_{t_{r-k}}^{(k)} &\lesssim \sum_{t_{r-k}=r-k}^{T-k}t_{r-k}^{r-k-1}\log T\log\log 4T\notag\\
&\lesssim \int_{0}^T x^{r-k-1}\log T\log\log 4T dx \lesssim T^{r-k}\log T\log\log 4T.
\end{align}
 
Viewing $\sigma_{t_{r-k}}^{(k)}$ and $\kappa_{t_{r-k}}^{(k)}$ (for $t_{r-k}$ from $r-k$ to $T-k$) as the $\sigma_i$ and $\kappa_i$ (for $i$ changing from $1$ to $T-r+1$) in Lemma \ref{lem:combn}, employing the bounds \eqref{eq:newhan-2} and \eqref{eq:newhan-3}, we deduce that there exist absolute constants $C_{21}$ and $C_{22}$ only depending on $\gamma$ and $r$ such that, for all $\eta$ satisfying 
\[
0<\eta<\big[C_{22} M T^{-1}(\log T)(\log\log 4T)\big]^{-1},
\]
we have, for $k=1,\ldots,r-1$,
\begin{align*}
\log\E\exp(\eta S_k)\lesssim&\frac{4C_2\eta^2M^2 T^{2k-2}\binom{T}{r}^{-2}T^{2r-2k+1}}{1-2C_{21}\eta MT^{k-1}\binom{T}{r}^{-1}T^{r-k}(\log T)(\log\log 4T)}\\
\lesssim&\frac{\eta^2M^2 T^{-1}}{1-C_{22}\eta M T^{-1}(\log T)(\log\log 4T)}.
\end{align*}
~\\

{\bf Step II.} Secondly, we focus on the case $k=r$. First we use the decomposition for $S_r$ to deduce
\begin{align*}
\log\E\exp(\eta S_r)=\log\E\exp\Big[\eta T^{r-1}\cdot\binom{T}{r}^{-1}\sum_{t_1=1}^{T-r+1}B_{[t_1:t_r]}^{[\cdot]}\Big].
\end{align*}
According to Lemma \ref{lem:combn} and Properties {\bf P1}, {\bf P2}, {\bf P3}, we deduce that there exists an absolute constant $C_{23}$ only depending on $\gamma$ and $r$ such that, for all $\eta$ satisfying 
\[
0<\eta<\big[C_{23} M T^{-1}(\log T)(\log\log 4T)\big]^{-1},
\]
we have
\begin{align*}
\log\E\exp(\eta S_r)\lesssim&\frac{4C_2\eta^2 T^{2r-2}\binom{T}{r}^{-2}(T - r + 1)}{1-2C_1\eta M\binom{T}{r}^{-1}T^{r-1}\log(T - r + 1)\log\log4(T - r + 1)}\\
 \lesssim&\frac{\eta^2M^2 T^{-1}}{1-C_{23}\eta M T^{-1}(\log T)(\log\log 4T)}.
\end{align*}
This completes the proof of the second step. 
~\\

In the end, we summarize the previous results to finalize the proof. According to the fact that
\begin{align*}
\log\E\Big[\exp\Big(\eta(U_r(X_1,\ldots,X_T)-\theta^*(h))\Big)\Big]=\log\E\Big[\exp\Big(\eta\sum_{k=1}^r S_k\Big)\Big],
\end{align*}
and the upper bounds for $\log\E\exp(\eta S_k)$, $k=1,\ldots,r$, we employ Lemma \ref{lem:combn} to conclude that there exists a positive constant $C_6$ only depending on $\gamma$ and $r$ such that, for all $\eta$ satisfying
\[
0<\eta<\big[C_6M T^{-1}(\log T)(\log\log 4T)\big]^{-1},
\]
we have
\begin{align*}
\log\E\Big[\exp\big(\eta(U_r(X_1,\ldots,X_T)-\theta^*(h))\big)\Big]\lesssim \frac{\eta^2M^2 T^{-1}}{1-C_6\eta M T^{-1}(\log T)(\log\log 4T)}.
\end{align*}

In terms of probabilities, there exists an absolute constant $C_5>0$, depending only on $\gamma$ and $r$, such that for all $x\geq 0$ and $T$ sufficiently large, we have
\begin{align*}
\P(|U_r(X_1,\ldots,X_T)-\theta^*(h)|\geq x)\leq 2\exp\Big(-\frac{C_5x^2T}{M^2+Mx(\log T)(\log\log 4T)}\Big).
\end{align*}

This completes the proof.
\end{proof}

\subsection{Bounding the bias term}
\begin{theorem}[Bias term]\label{thm:bias} Under Assumptions {\bf (A1)} and {\bf (A3)}, we have
\[
|\theta^*(h)-\theta(h)|=O(M/\sqrt{T}).
\]
\end{theorem}
\begin{proof}
The proof is largely the same to \cite{yoshihara1976limiting}, and is accordingly omitted. 
\end{proof}

\appendix

\section{Proof of Proposition \ref{prop:3}}

\begin{proof}
Since, for all $n\geq 1$, $j\geq1$, and any subset $S_t\subset G_t$, we have
\begin{align*}
&\Big|\P(X_{t_J+j+n}\in S_{t_J+j+n},\ldots|X_{t_1}\in S_{t_1},\ldots,X_{t_J}\in S_{t_J},X_{t_J+1}\in S_{t_J+1},\ldots,\\
&~~~~~~X_{t_J+j}\in S_{t_J+j})-\P(X_{t_J+j+n}\in S_{t_J+j+n},\ldots|X_{t_1} \in  S_{t_1},\ldots,X_{t_J}\in S_{t_J})\Big|\\
\leq&\Big|\P(X_{t_J+j+n} \in  S_{t_J+j+n}, \ldots |X_{t_1} \in S_{t_1}, \ldots ,X_{t_J} \in S_{t_J},X_{t_J+1} \in S_{t_J+1},\ldots,\\
&~~~~~~X_{t_J+j}\in S_{t_J+j})-\P(X_{t_J+j+n} \in  S_{t_J+j+n},\ldots)\Big|+\\
&\Big|\P(X_{t_J+j+n} \in  S_{t_J+j+n},\ldots|X_{t_1}\in S_{t_1},\ldots,X_{t_J}\in S_{t_J})-\P(X_{t_J+j+n}\in S_{t_J+j+n},\ldots)\Big|\\
\lesssim& 2\exp(-Cn)\lesssim \exp(-Cn),
\end{align*}
we conclude that
\begin{align*}
&\sup_{j\geq1}\Big|\P(X_{t_J+j+n}\in S_{t_J+j+n},\ldots|X_{t_1}\in S_{t_1},\ldots,X_{t_J}\in S_{t_J},X_{t_J+1}\in S_{t_J+1},\ldots,\\
&~~~~~~~~~X_{t_J+j} \in  S_{t_J+j})-\P(X_{t_J+j+n}\in S_{t_J+j+n},\ldots|X_{t_1}\in S_{t_1},\ldots,X_{t_J}\in S_{t_J})\Big|\\
\lesssim& \exp(-Cn).
\end{align*}
This proves that, uniformly, for any integer $J$ such that $1\leq J\leq r-1$ and arbitrary $1\leq t_1<\ldots<t_J\leq T$, conditional on $X_{t_1},\ldots,X_{t_J}$, the sequence $\{X_t\}_{t=t_J+1}^{\infty}$ satisfies the $\phi$-mixing condition in \eqref{eq:phi} uniformly. 
\end{proof}

\section{Proof of Proposition \ref{prop:new1}}
It is obvious that proving Proposition \ref{prop:new1} is equivalent to proving the following statement. 
\begin{lemma}
Suppose $X,Y,Z$ are three random vectors defined on the probability space $(\Omega_1,\cF_1,\P_1)$, $(\Omega_2,\cF_2,\P_2)$, and $(\Omega_3,\cF_3,\P_3)$, with each mapped to a vector-valued space $(\reals^{k_1},\cB(\reals^{k_1}))$, $(\reals^{k_2},\cB(\reals^{k_2}))$, and $(\reals^{k_3},\cB(\reals^{k_3}))$. Here $k_1\in\N$, $k_2,k_3\in\N\cup\{\infty\}$, and $\cB(\reals^{k_i})$ is the Borel algebra of $\reals^{k_i}$. Define $\mu_{123}$ to be the Lebesgue measure on $\reals^{k_1}\times\reals^{k_2}\times\reals^{k_3}$. Similarly, define $\mu_{12},\mu_{13},\mu_{23},\mu_{1},\mu_{2},\mu_{3}$ to be the Lebesgue measures on $\reals^{k_1}\times\reals^{k_2},\reals^{k_1}\times\reals^{k_3},\reals^{k_2}\times\reals^{k_3},\reals^{k_1},\reals^{k_2},\reals^{k_3}$. Suppose the following two properties hold:
\begin{itemize}
\item
$(X,Y,Z)$ has a joint density function $f_{123}(w_1,w_2,w_3): \reals^{k_1}\times\reals^{k_2}\times\reals^{k_3}\rightarrow \reals$; 
\item
For any sets $A_1\in\cB(\reals^{k_1}), A_2\in\cB(\reals^{k_2})$ with $\int_{A_1\times A_2}f_{12}(w_1,w_2)d\mu_{12}> 0$, and arbitrary set $A_3\in\cB(\reals^{k_3})$, 
\begin{align}\label{eq:cont1}
\Big| \frac{\int_{A_1\times A_2\times A_3}f_{123}(w_1,w_2,w_3)d\mu_{123}}{\int_{A_1\times A_2}f_{12}(w_1,w_2)d\mu_{12}}- \int_{A_3}f_3(w_3)d\mu_3\Big| \leq \zeta,
\end{align}
where $f_{12}, f_{13}, f_{23}, f_1, f_2, f_3$ are the density functions corresponding to the random vectors $(X,Y), (X,Z), (Y,Z), X, Y, Z$, and $\zeta\geq 0$ is a fixed constant.
\end{itemize}
We then have, for any $w_1^0\in \cW_1^0$, any $A_2\in\cB(\reals^{k_2})$ with $\int_{A_2}f_{12}(w_1^0,w_2)d\mu_2> 0$, and any $A_3\in\cB(\reals^{k_3})$,
\[
\Big| \frac{\int_{A_2\times A_3}f_{123}(w_1^0, w_2,w_3)d\mu_{23}}{\int_{A_2}f_{12}(w_1^0,w_2)d\mu_2}- \int_{A_3}\frac{f_{13}(w_1^0,w_3)}{f_1(w_1^0)}d\mu_3\Big| \leq 2\zeta,
\]
where $\cW_1^0$ belongs to $\{w_1\in\reals^{k_1}: f_1(w_1)\ne0\}$ and satisfies $\P_1(\cW_1^0)=1$.
\end{lemma}
\begin{proof}
For notation simplicity, we write $f_{123}(\cdot)$ as $f(\cdot)$ in the sequel. By the definition of the density functions, for proving the lemma, it is equivalent to showing
\begin{align*}
\sup_{w_1^0\in\cW_1^0}\Big|\frac{\int_{A_2\times A_3}f(w_1^0,w_2,w_3)d\mu_{23}}{\int_{A_2\times \reals^{k_3}}f(w_1^0,w_2,w_3)d\mu_{23}}-\frac{\int_{\reals^{k_2}\times A_3}f(w_1^0,w_2,w_3)d\mu_{23}}{\int_{\reals^{k_2}\times \reals^{k_3}}f(w_1^0,w_2,w_3)d\mu_{23}}\Big| \leq 2\zeta.
\end{align*}
We further have, for each $w_1^0$,
\begin{align*}
&\Big|\frac{\int_{A_2\times A_3}f(w_1^0,w_2,w_3)d\mu_{23}}{\int_{A_2\times \reals^{k_3}}f(w_1^0,w_2,w_3)d\mu_{23}}-\frac{\int_{\reals^{k_2}\times A_3}f(w_1^0,w_2,w_3)d\mu_{23}}{\int_{\reals^{k_2}\times \reals^{k_3}}f(w_1^0,w_2,w_3)d\mu_{23}}\Big| \\
=&\Big|\underbrace{\Big[\frac{\int_{A_2\times A_3}f(w_1^0,w_2,w_3)d\mu_{23}}{\int_{A_2\times \reals^{k_3}}f(w_1^0,w_2,w_3)d\mu_{23}}-\int_{A_3}f_3(w_3)d\mu_3\Big]}_{D_1}-\underbrace{\Big[\frac{\int_{\reals^{k_2}\times A_3}f(w_1^0,w_2,w_3)d\mu_{23}}{\int_{\reals^{k_2}\times \reals^{k_3}}f(w_1^0,w_2,w_3)d\mu_{23}}-\int_{A_3}f_3(w_3)d\mu_3\Big]}_{D_2}\Big|.
\end{align*}
Now we turn to bound $D_1$ and $D_2$ respectively. For $D_1$, we aim to show
\begin{align*}
D_1\leq \zeta.
\end{align*}
Let us define, for fixed $w_1\in\reals^{k_1}$,
\begin{align*}
F_1(w_1):=\int_{A_2\times A_3}f(w_1,w_2,w_3)d\mu_{23}~~~{\rm and}~~~F_2(w_1):=\int_{A_2\times \reals^{k_3}}f(w_1,w_2,w_3)d\mu_{23}.
\end{align*}
According to \eqref{eq:cont1} and the triangle inequality
\begin{align*}
D_1\leq\Big|\frac{\int_{A_1}F_1(w_1)d\mu_1}{\int_{A_1}F_2(w_1)d\mu_1}-\frac{F_1(w_1^0)}{F_2(w_1^0)}\Big|+\Big|\frac{\int_{A_1}F_1(w_1)d\mu_1}{\int_{A_1}F_2(w_1)d\mu_1}-\int_{A_3} f(w_3)d\mu_3\Big|,
\end{align*}
it suffices to show for any $\epsilon>0$, there exists a subset $A_1(\epsilon;w_1^0)\in\cB(\reals^{k_1})$ such that 
\begin{align}\label{eq:cont2}
\Big|\frac{\int_{A_1(\epsilon;w_1^0)}F_1(w_1)d\mu_1}{\int_{A_1(\epsilon;w_1^0)}F_2(w_1)d\mu_1}-\frac{F_1(w_1^0)}{F_2(w_1^0)}\Big|\leq \epsilon
\end{align}
and $\int_{A_1(\epsilon;w_1^0)}F_2(w_1)d\mu_1>0$. 
Using Theorem 2.3.8 and Corollary 2.2.2(b) in \cite{ash2000probability}, for $w_1^0$ over a set of probability one (up to $\P_1$), for arbitrary $\epsilon>0$ (not depending on $w_1^0$), we could pick a set $A\in\cB(\reals^{k_1})$ surrounding $w_1^0$ such that
\[
\Big|\frac{\int_{A}F_1(w_1)d\mu_1}{\mu_1(A)}-F_1(w_1^0)\Big|\leq \epsilon.
\]
A similar bound applies to $F_2(w_1^0)$. Accordingly, for $w_1^0$ over a set of probability one (up to $\P_1$), for arbitrary $\epsilon_1,\epsilon_2>0$ (not depending on $w_1^0$), we could pick $A_{1,\epsilon_1,\epsilon_2; w_1^0}$ surrounding $w_1^0$ such that
\begin{align*}
\frac{F_1(w_1^0)-\epsilon_1}{F_2(w_1^0)+\epsilon_2}&\leq \frac{\int_{A_{1,\epsilon_1,\epsilon_2; w_1^0}}F_1(w_1)d\mu_1\Big/\mu_1(A_{1,\epsilon_1,\epsilon_2;w_1^0})}{\int_{A_{1,\epsilon_1,\epsilon_2; w_1^0}}F_2(w_1)d\mu_1\Big/\mu_1(A_{1,\epsilon_1,\epsilon_2;w_1^0})}=\frac{\int_{A_{1,\epsilon_1,\epsilon_2; w_1^0}}F_1(w_1)d\mu_1}{\int_{A_{1,\epsilon_1,\epsilon_2; w_1^0}}F_2(w_1)d\mu_1}\leq \frac{F_1(w_1^0)+\epsilon_1}{F_2(w_1^0)-\epsilon_2}.
\end{align*}
For a suitable choice of $(\epsilon_1,\epsilon_2)$, we can have
\begin{align*}
\frac{F_1(w_1^0)}{F_2(w_1^0)}-\epsilon\leq \frac{F_1(w_1^0)-\epsilon_1}{F_2(w_1^0)+\epsilon_2}<\frac{F_1(w_1^0)+\epsilon_1}{F_2(w_1^0)-\epsilon_2}\leq \frac{F_1(w_1^0)}{F_2(w_1^0)}+\epsilon
\end{align*}
and $\int_{A_{1,\epsilon_1,\epsilon_2,w_1^0}}F_2(w_1)d\mu_1>0$. Setting $A_1(\epsilon;w_1^0)$ in \eqref{eq:cont2} to be $A_{1,\epsilon_1,\epsilon_2; w_1^0}$ , we have \eqref{eq:cont2} holds,  and hence by \eqref{eq:cont1},
\begin{align*}
D_1\leq \zeta+\epsilon.
\end{align*}
Since $\epsilon$ here is any positive constant, we obtain the desirable result for $D_1$. Similarly, we have the same bound for $D_2$. According to the triangle inequality $|D_1-D_2|\leq |D_1|+|D_2|$, we have $|D_1-D_2|\leq 2\zeta$, and thus complete the proof.
\end{proof}

\section{Auxiliary lemmas}



The following lemma, regarded as a strengthened version of the Jensen's inequality, comes from \cite{merlevede2009bernstein}.

\begin{lemma}\label{lem:combn}
Let $Z_1,Z_2,\ldots$ be a sequence of real valued random variables. Assume that there exist positive constants $\sigma_1,\sigma_2,\ldots$ and $\kappa_1,\kappa_2,\ldots$ such that, for any $i\geq1$ and any $\eta$ in $[0,1/\kappa_i)$,
\begin{align*}
\log\E[\exp(\eta Z_i)]\leq(\sigma_i\eta)^2/(1-\kappa_i\eta).
\end{align*}
Then, for any positive $n$ and any $\eta$ in $[0,1/(\kappa_1+\ldots+\kappa_n))$, we have
\begin{align*}
\log\E[\exp(\eta(Z_1+\ldots+Z_n))]\leq(\sigma \eta)^2/(1-\kappa \eta),
\end{align*}
where $\sigma=\sigma_1+\ldots+\sigma_n$ and $\kappa=\kappa_1+\ldots+\kappa_n$.
\end{lemma}

The following lemma gives a rigorous proof of a fundamental conditional independence property. Although this statement must have been shown in other places, we fail to locate one, and decide to provide a proof by ourselves. 

\begin{lemma}\label{lem:condind}
Suppose we have three random variables $X$, $Y$, and $Z$ mapping to $(\cW_1,\cB_{\cW_1}), (\cW_2,\cB_{\cW_2}),$ and $(\cW_3,\cB_{\cW_3})$ respectively. If $X$ and $(Y,Z)$ are independent, we have $X$ and $Y$ are independent conditional on $Z$.
\end{lemma}
\begin{proof}
Define the sigma field generated by $Z$ to be $\sigma(Z)$, and $\mathbf{1}(\cdot)$ to be the indicator function. Similarly, define the sigma field generated by $Y$ and $Z$ to be $\sigma(Y,Z)$. 
We have, for any set $A\in\cB_{\cW_1}$,
\begin{align}\label{eq:cond1}
\E[\mathbf{1}(X\in A)|\sigma(Y,Z)]=\E[\mathbf{1}(X\in A)]=\E[\mathbf{1}(X\in A)|\sigma(Z)],~~{\rm a.s.}.
\end{align}
It suffices to show that, for any sets $A\in\cB_{\cW_1}$, $B\in\cB_{\cW_2}$, we have
\begin{align}\label{eq:cond0}
\E[\mathbf{1}(X\in A)\mathbf{1}(Y\in B)|\sigma(Z)]=\E[\mathbf{1}(X\in A)|\sigma(Z)]\cdot\E[\mathbf{1}(Y\in B)|\sigma(Z)],~~{\rm a.s.}.
\end{align}
By the definition of conditional expectation and \eqref{eq:cond1}, we have for any set $G\in\sigma(Z)$,
\begin{align*}
\int_{G\cap Y^{-1}(B)}\mathbf{1}(X\in A)d\P&=\int_{G\cap Y^{-1}(B)}\E[\mathbf{1}(X\in A)|\sigma(Y,Z)]d\P=\int_{G\cap Y^{-1}(B)}\E[\mathbf{1}(X\in A)|\sigma(Z)]d\P.
\end{align*}
Using the definition of conditional expectation again, we have
\begin{align}\label{eq:cond2}
\int_{G\cap Y^{-1}(B)}\mathbf{1}(X\in A)d\P&=\int_G \E[\mathbf{1}(X\in A)|\sigma(Z)]\mathbf{1}(Y\in B)d\P\notag\\
&=\int_G \E\Big[\E[\mathbf{1}(X\in A)|\sigma(Z)]\mathbf{1}(Y\in B)\Big|\sigma(Z)\Big]d\P. 
\end{align}
Since $\E[\mathbf{1}(X\in A)|\sigma(Z)]$ is measurable with respect to $\sigma(Z)$, we have
\begin{align}\label{eq:cond3}
\E\Big[\E[\mathbf{1}(X\in A)|\sigma(Z)]\mathbf{1}(Y\in B)\Big|\sigma(Z)\Big]=\E[\mathbf{1}(X\in A)|\sigma(Z)]\E[\mathbf{1}(Y\in B)|\sigma(Z)],~~{\rm a.s.}.
\end{align}
Combining \eqref{eq:cond2} and \eqref{eq:cond3}, we further have
\begin{align*}
\int_{G\cap Y^{-1}(B)}\mathbf{1}(X\in A)d\P=\int_G \E[\mathbf{1}(X\in A)|\sigma(Z)]\E[\mathbf{1}(Y\in B)|\sigma(Z)]d\P,
\end{align*}
which indicates that, for any set $G\in\sigma(Z)$,
\begin{align*}
\int_G\mathbf{1}(X\in A)\mathbf{1}(Y\in B)d\P=\int_G \E[\mathbf{1}(X\in A)|\sigma(Z)]\E[\mathbf{1}(Y\in B)|\sigma(Z)]d\P.
\end{align*}
Combining the above equation and the fact that $\E[\mathbf{1}(X\in A)|\sigma(Z)]\E[\mathbf{1}(Y\in B)|\sigma(Z)]$ is measurable with respect to $\sigma(Z)$, 
we succeed in proving \eqref{eq:cond0}, and thus complete the proof.
\end{proof}

\section*{Acknowledgement}

The author is grateful for the helpful discussions with Drs. Han Liu and Jianqing Fan, which motivated this work. The author thanks Sheng Xu for discussions and the proof of Lemma \ref{lem:condind}. 

\bibliographystyle{apalike}
\bibliography{mix}

\end{document}